\documentclass[12pt,oneside]{article}

\usepackage[small,nohug,heads=vee]{diagrams}
\diagramstyle[labelstyle=\scriptstyle]

\usepackage{amssymb}

\usepackage{amsmath}

\usepackage{amsfonts}

\usepackage{longtable}

\newtheorem{theorem}{Theorem}[section]
\newtheorem{lemma}[theorem]{Lemma}

\newtheorem{definition}[theorem]{Definition}

\newenvironment{ack}
{\begin{trivlist}  \item \textsc{Acknowledgments}~} {\end{trivlist}}

\newenvironment{proof}
{\begin{trivlist}  \item \textsc{Proof:}~} {\hfill $\Box$
\end{trivlist}}

\newenvironment{proof of claim}
{\begin{trivlist}  \item \textsc{Proof of Claim:}~} {\hfill $\Box$
\end{trivlist}}

\newcommand{\closure}[1]{\ensuremath{\mathrm{cl}}(#1)}

\def \cl {\operatorname{cl}}
\def \st {\operatorname{st}}

\def \ind{\operatorname{ind}}

\newcommand{\im}[1]{\ensuremath{\mathrm{im}}(#1)}
\newcommand{\bd}[1]{\ensuremath{\mathrm{bd}}(#1)}

\newcommand{\ma}{\mathfrak{m}}

\def \cl {\mathrm{cl}}

\begin{document}

\author{Jana Ma\v{r}\'{i}kov\'{a}\\
Dept. of Math., Western Illinois
University\footnote{Part of this work was done while the author was a Fields Ontario postdoctoral fellow, during the thematic program on o-minimality and real analytic geometry at the Fields Institute, Toronto, spring 2009.}\\
J-Marikova@wiu.edu}
\title{O-minimal residue fields of o-minimal fields}

\maketitle

\begin{abstract}  Let $R$ be an o-minimal field with
a proper convex subring $V$.  We axiomatize the class of all structures $(R,V)$ such that $\boldsymbol{k}_{\ind}$, the corresponding residue field with structure induced from $R$ via the residue map, is o-minimal.  More precisely, in
\cite{thesispaper} it was shown that certain first order conditions on $(R,V)$ are sufficient for the o-minimality of $\boldsymbol{k}_{\ind}$.  Here we prove that these conditions are also necessary.
\end{abstract}

\begin{section}{Introduction}
Throughout this paper we let $R$ be an o-minimal field, that is, an o-minimal expansion of a real
closed field, and $V$ a proper convex subring (hence a valuation ring).  Let
$\st \colon V \to \boldsymbol{k}$ be the corresponding residue
(standard part) map with kernel $\ma$ and residue field
$\boldsymbol{k}= V / \ma$.  For $X\subseteq R^n$ we set $\st X := \st (X \cap V^n)$.  By
$\boldsymbol{k}_{\ind}$ we denote the residue field expanded by all sets
$\st X \subseteq \boldsymbol{k}^n$ with definable
$X \subseteq R^n$, for some $n$. ``Definable'' means
``definable (with parameters) in $R$'', unless indicated otherwise.

\smallskip\noindent
Here are some notational conventions we use: By
$\mathbb{N} = \{ 0,1,\dots \}$ we denote the set of natural numbers, and
the letters $i,j,k,l,m,n$ denote natural numbers.  For $a\in R^n$
and $i= 1,\dots ,n$, we let $a_i$ be the $i$-th coordinate
of $a$, unless stated otherwise.  We let $I:=[-1,1]\subseteq R$, and for a definable set $X \subseteq R^{1+n}$ and $r \in R$, we
set
$$X(r):= \{ x \in R^n \colon \; (r,x) \in X \}.$$
By $d(x,y)$ we denote the euclidean distance between $x$
and $y$. For definable $Y \subseteq R^n$ and $x\in R^n$,
$$d(x,Y):=\inf \{ d(x,y):\; y\in Y  \}.$$ If $X,Y \subseteq R^n$ are definable, then $X\cong Y$ means that there is a definable homeomorphism $X \to Y$.
For a definable $C^1$-map $\phi \colon X \to R^n$, with $X$ an open subset of $R^m$, and $a\in X$, we denote by $D \phi (a)$ the Jacobian matrix of $\phi$ at $a$.
 If $X \subseteq R^n$, then we denote by $\closure{X}$ the closure of $X$ in $R^n$, by $X^o$ the interior of $X$ in $R^n$, and we let $\partial X:= \closure{X}\setminus X$, and $\bd{X}:= \closure{X} \setminus X^o$.
For $0<k < n$, $p^{n}_{k}$ denotes the projection map $$R^n \to R^k \colon (x_1 , \dots , x_n ) \mapsto (x_1 , \dots ,x_k ).$$  For a definable map $\phi \colon X \to R^n$, $X \subseteq R^m$, and $i=1,\dots ,m$, we denote by $\phi_i$ the $i$-th coordinate function of $\phi$.  If $a,b \in R^n$, then $[a,b]$ denotes the set $\{ (t-1)a+tb:\; t\in [0,1] \}$, and $(a,b)$ denotes the set $\{ (1-t)a+tb:\; t\in (0,1) \}$.


\smallskip\noindent
We recall now some definitions and the main result from \cite{thesispaper}:

\begin{trivlist}
\item $(R,V) \models \Sigma (n)$ means that for every definable $X \subseteq I^{1+n}$ there is $\epsilon_0 \in \ma^{>0}$ such that $\st X(\epsilon_0 ) = \st X(\epsilon )$ for all $\epsilon \in \ma^{>\epsilon_0 }$.
\item $(R,V)\models \Sigma$ means that $(R,V) \models \Sigma (n)$ for every $n$.
\end{trivlist}

\begin{theorem}[\cite{thesispaper}]\label{sigmaimplkindomin} If $(R,V) \models \Sigma$, then
$\boldsymbol{k}_{\ind}$ is o-minimal, and the subsets of $\boldsymbol{k}^n$ definable in $\boldsymbol{k}_{\ind}$ are exactly the finite unions of sets $\st X \setminus \st Y$, where $X,Y \subseteq R^n$ are definable.
\end{theorem}
The question, whether, on the other hand, o-minimality of
$\boldsymbol{k}_{\ind}$ implies $(R,V) \models \Sigma$, was left
unanswered.  Here, we give a positive answer to it,  yielding the remarkable fact, that the o-minimality of $\boldsymbol{k}_{\ind}$ is equivalent to $(R,V)$ satisfying a first-order axiom scheme.
More concretely, we obtain:

\begin{theorem}\label{maintheorem}
The following conditions on $(R,V)$ are equivalent:
\begin{enumerate}
\item $\boldsymbol{k}_{\ind}$ is o-minimal

\item $\Sigma (1)$

\item $\Sigma$.
\end{enumerate}
\end{theorem}
The implication $3\Rightarrow 1$ is Theorem \ref{sigmaimplkindomin}, and $1 \Rightarrow 2$ can be found in \cite{jalou}.  So it is left to show $\Sigma (1) \Rightarrow \Sigma$.
Note that by cell decomposition, and using
the definable homeomorphism
$$\tau \colon R^2 \rightarrow (-1,1)^{2}\colon (x_1 ,x_2 ) \mapsto
(\frac{x_1 }{\sqrt{1 + x_{1}^2}} , \frac{x_2
}{\sqrt{1+x_{2}^{2}}}),$$
$\Sigma (1)$ is equivalent to:
for every definable $f\colon R \to R$ there is $\epsilon_0 \in \ma^{>0}$ such that $\st f(\epsilon_0 ) = \st f(\epsilon )$ for all $\epsilon \in \ma^{>\epsilon_0}$.

 Towards proving $\Sigma (1) \Rightarrow \Sigma$ we use a reduction to definable families of one-dimensional subsets of $I^2$ from \cite{jalou} (where it is stated for the case when $\boldsymbol{k}_{\ind}$ is o-minimal, but the proof does not use this assumption):
\begin{lemma}[\cite{jalou}]\label{red to dim2}
Let $n\geq 1$, and suppose that for all definable $X\subseteq I^{1+n}$ with $\dim X(r) \leq 1$ for all $r\in I$ there is $\epsilon_0 \in \ma^{>0}$ such that $\st X (\epsilon_0 ) = \st X (\epsilon )$ for all $\epsilon \in \ma^{ > \epsilon_0}$. Then $(R,V) \models \Sigma (n)$.
\end{lemma}
The above lemma will enable us to work mainly in the plane, but at the cost of having to handle level curves of definable functions.  The main result on level curves is Lemma \ref{main lemma}, which is proved in the first section.  Its proof uses
the (two-variable version of the) Invariance of Domain Theorem for o-minimal fields by Woerheide (see \cite{woerheide}):
\begin{theorem}[\cite{woerheide}]\label{invarofdom}
Every injective definable, continuous map $f\colon X \to R^n$, where $X\subseteq R^n$, is open.
\end{theorem}
 We remark that it is well-known that $\boldsymbol{k}_{\ind}$ is always weakly o-minimal (this follows from a result by Baisalov and Poizat in \cite{bp}).  An example to the effect that $\boldsymbol{k}_{\ind}$ is not always o-minimal is given in \cite{thesispaper}.  Results of van den Dries and Lewenberg (see \cite{tconv}) show that $\boldsymbol{k}_{\ind}$ is o-minimal if $V$ is $T$-convex, where $T$ is the theory of $R$. Hrushovski, Peterzil and Pillay observe in \cite{nip} that $\boldsymbol{k}_{\ind}$ is o-minimal when $R$ is sufficiently saturated and $V$ is the convex hull of $\mathbb{Q}$ in $R$.

\smallskip\noindent
Here is a related open question:
Assume that $(R,V)\models \Sigma (1)$.  Are the sets $\st X$, with $X\subseteq R^n$ definable in $(R,V)$, definable in $\boldsymbol{k}_{\ind}$?  (A positive answer was given in \cite{thesispaper} for the special case when $R$ is $\omega$-saturated and $V$ is the convex hull of $\mathbb{Q}$ in R.)

\end{section}

\begin{section}{Level curves of definable functions}
\begin{definition}
Let $f\colon X \to R$ be definable with $X \subseteq
R^n$. We say that $\epsilon_0$ is $good$ $for$ $f$ if $\epsilon_0
\in \ma^{>0}$ and $\st f^{-1}(\epsilon_0 ) = \st f^{-1}(\epsilon )$
for all $\epsilon \in \ma^{> \epsilon_0 }$.
\end{definition}

\begin{lemma}\label{good for dim 1}
Suppose $(R,V) \models \Sigma (1)$, and let $f\colon X \to R$ be
definable, with $X\subseteq R^n$ of
dimension one. Then there is $\epsilon_0$ good for $f$.  \end{lemma}
\begin{proof}
The
case when there is no $\delta \in \ma^{>0}$ and $q>\ma$ such that
$[\delta ,q] \subseteq \im{f}$ is trivial. So suppose
$\delta \in \ma^{>0}$ and $q> \ma$ are such that $[\delta ,q]\subseteq
\im{f}$.  Then $f^{-1}([\delta ,q])$ has finitely many definably
connected components $X_1 , \dots ,X_k$.  After possibly shrinking $[\delta ,q]$ subject to the conditions $\delta \in \ma^{>0}$ and $q>\ma$, we may assume
that for every $i$, either $f|_{X_i }$ is injective or constant.

If $f|_{X_i}$ is constant and $f(x) \in \ma^{>0}$, where $x\in X_i$, then
set $\epsilon_i := f(x)$.  If $f|_{X_i}$ is constant and $f(x) \not\in \ma^{>0}$, for $x\in X_i$, then set $\epsilon_i =0$. If $f|_{X_i}$ is injective, then $f|_{X_i}$ has
an inverse,
$f^{-1}$.  Apply $\Sigma (1)$ to each coordinate function $(f^{-1})_j$ of
$f^{-1}$ to obtain a positive infinitesimal $\epsilon_{ij}$ such that $\st f^{-1}(\epsilon )= \st
f^{-1}(\epsilon_{ij} )$ for every $\epsilon \in \ma^{> \epsilon_{ij} }$.
Then
put $\epsilon_i := \max_{j\in \{1,\dots ,n  \}} \epsilon_{ij}$.  Now
set
$$\epsilon_0 :=\max_{i=1,\dots ,k}\{ \epsilon_i  \}.$$
It is easy to see that $\epsilon_0$ is good for $f$.
\end{proof}
The following lemma is obvious.
\begin{lemma}\label{partition}
 Let $f\colon X \to R$, where $X \subseteq R^n$, be definable, and let $X_1 , \dots ,X_k$ be a partition of $X$ into definable sets.  If $\epsilon_1 , \dots , \epsilon_k$ are good for $f|_{X_1 }, \dots , f|_{X_k }$, then $\max \{ \epsilon_1 , \dots , \epsilon_k   \}$ is good for $f$.
\end{lemma}

The proof of the next lemma is a little lengthy, but the idea is fairly straight-forward, so here's a rough outline:  In order to find $\epsilon_0$ good for $G\colon X\to R^{\geq 0}$ as in the statement of Lemma \ref{main lemma}, we first show that, essentially, the level curves of $G$ are all uniformly definably homeomorphic to $I$.  Then we find a definable curve $\gamma$ passing through the points of $X$ where $|\nabla G|$ is minimal (o-minimality enables us to reduce to the case when $G$ is sufficiently nice).  We use $\Sigma (1)$ to find $\epsilon_0$ good for $G|_{\im{\gamma}}$, and for $G |_{\bd{X}}$.  Then we proceed to show that this $\epsilon_0$ is also good for $G$.  We assume towards a contradiction that it is not, hence some $a\in G^{-1}(\epsilon)$, where $\epsilon \in \ma^{> \epsilon_0 }$, is at noninfinitesimal distance from $G^{-1}(\epsilon_0 )$.  We would be done if this assumption would enable us to find a definable curve in $X$, of length $>\ma$, such that at every point of the curve, the tangent vector to this curve points in the direction of $\nabla G$, and such that the image of $G$ restricted to the curve is $[\epsilon_0 , \epsilon ]$.  This would put us into the paradoxical situation that the steepest path down a hill is much longer than a path which is less steep.  To find a curve as described above is not necessarily possible, but we define a curve $\alpha$ (long, steep path down the hill) with the above properties, except that, at every point $x$ of $\alpha$, the tangent vector to $\alpha$ is only reasonably close to $\nabla G(x)$.

\begin{lemma}\label{main lemma}  Suppose $(R,V)\models \Sigma (1)$, and let
$X\subseteq V^2$ be a closed definable set of dimension two, and let $G\colon X \to R^{\geq 0}$ be definable and continuous. Then there is $\epsilon_0$ good for $G$.
\end{lemma}
\begin{proof}
The points $(0,1)$, $(\frac{1}{\sqrt{2}},\frac{1}{\sqrt{2}})$, $(1,0)$,
$(\frac{1}{\sqrt{2}},-\frac{1}{\sqrt{2}})$, $(0,-1)$,
$(-\frac{1}{\sqrt{2}},-\frac{1}{\sqrt{2}})$, $(-1, 0)$, and
$(-\frac{1}{\sqrt{2}},\frac{1}{\sqrt{2}})$
determine a partition of $S^1 \subseteq R^2$ into
definable $S_1 , \dots ,S_{16}$, such that $S_i \cong (0,1)$ for
$i=1,\dots ,8$, and
$S_i$ is a singleton for $i=9,\dots ,16$.

By cell decomposition and Lemmas \ref{good for dim 1} and \ref{partition}, we may assume that
$X$ is the closure of a cell of dimension two such that $G$ is $C^1$ on $X^o$, and
either $|\nabla G|=0$ on $X^o$, or $|\nabla G|>0$ on $X^o$ and
there is $i\in \{1, \dots ,8\}$ satisfying
$-\frac{\nabla G(x)}{|\nabla G(x)|} \in S_i$ for all $x \in X^o$.  If $|\nabla G|=0$ on $X^o$, then $G$ is constant on $X$, so we only need to consider the second option.

The case when $G(X)$ does not contain an interval $[\delta ,q]$,
with $\delta \in \ma^{>0}$ and $q>\ma$ is trivial.  So let $\delta
\in \ma^{>0}$ and $q>\ma$ be such that $[\delta ,q]\subseteq G(X)$.
 Using the Trivialization Theorem (1.7, p.147 in \cite{book}), and after possibly shrinking $[\delta ,q]$, subject to $\delta \in
\ma^{>0}$ and $q>\ma$, and replacing $X$ by $G^{-1}[\delta ,q]$, we obtain a definable homeomorphism $$h \colon X\to [\delta ,q] \times G^{-1}(\delta )$$ such that the diagram

\begin{diagram}
X &      & \rTo(2,0)^{h} & & [\delta , q]\times G^{-1}(\delta )\\
  &\rdTo & & \ldTo &\\
  &&  [\delta ,q] &&
\end{diagram}
commutes, where $X\to [\delta ,q]$ is given by $G$ and $[\delta , q]\times G^{-1}(\delta ) \to [\delta ,q]$ is the projection map on the first factor.  So all $G^{-1}(t)$, for $t\in [\delta ,q]$, are uniformly definably homeomorphic, and we may assume that $\dim G^{-1}(\delta )=1$.

By Lemma \ref{partition},
we can reduce to the case when $X$ is definably connected.

\smallskip\noindent
{\bf Claim 1. \/}$G^{-1}(\delta )$ is definably connected.

 The set $X$ is closed and definably homeomorphic
to $G^{-1}(\delta ) \times I$. Assume towards a contradiction that
$G^{-1}(\delta )=Y_1 \dot\cup Y_2$, where $Y_1 , Y_2$ are definable,
nonempty and closed in $G^{-1}(\delta )$. Then $Y_1 ,Y_2$ are also closed in
$X$, and
$$X \cong (Y_1 \dot\cup Y_2 ) \times I \cong (Y_1 \times I)
\dot\cup (Y_2 \times I),$$ a
contradiction with $X$ being definably connected.

\smallskip\noindent
For $x=(x_1 , x_2 )\in V^2$ and $r>0$ we let $$B_{r}(x):=[x_1 -r, x_1 + r]\times [x_2 - r, x_2 + r].$$

\smallskip\noindent
{\bf Claim 2. \/}Let $Y \subseteq V^2$ be definable, definably connected, and closed.
 If for every $y\in Y$, there is $r>0$ such that $B_r (y) \cap Y \cong I$, then either $Y \cong S^1$ or $Y \cong I$.

 Let $\mathcal{S}$ be a stratification of $Y$.  Then $Y$ is a finite union of 0-dimensional cells $p_1 , \dots , p_k \in \mathcal{S}$, and 1-dimensional cells $q_1 , \dots , q_l \in \mathcal{S}$, where $k\geq 2$ and $l\geq 1$.  Every $q_i$ is definably homeomorphic to an open interval, and with every $q_i$, $Y$ contains its endpoints among $p_1 , \dots ,p_k$.  On the other hand, every $p_j$ is the endpoint of one or at most two of the $q_i$'s.  Now use the fact that $Y$ is definably connected.

\smallskip\noindent
{\bf Claim 3. \/}$G^{-1}(\delta ) \cong S^1$
or $G^{-1}(\delta ) \cong I$.

It suffices to prove $G^{-1}(t) \cong S^1$ or $G^{-1}(t) \cong I$ for $t\in (\delta ,q)$.  So let $t\in (\delta ,q)$, and assume towards a contradiction that
$G^{-1}(t) \not\cong S^1$ and $G^{-1}(t)\not\cong I$.
By Claims 1 and 2, we can find $a\in G^{-1}(t)$ and $r_0 >0$ with $G^{-1}(t) \cap B_{r}(a)
\not\cong I$ for
all $r \in (0,r_0 )$.
Hence, for all $r \in (0,r_0 )$, $G^{-1}(t) \cap B_{r}(a)$
is definably homeomorphic to the union of $k\geq 3$ closed line
segments $l_1 , \dots , l_k$ such that each $l_i$ has $a$ as an endpoint, and $\bigcap_{i=1}^{k} l_i = \{ a \}$ (we
identify $a$ here with its homeomorphic image). 
So let $r\in (0,r_0 )$, and let $$\phi \colon (\bigcup_{i=1}^{k}l_i )\times I \to X\cap B_r (a)$$ be a definable homeomorphism.
We can  map $(-1,1)^2$ into $\phi ((l_1 \cup l_2 ) \times I)$ via a definable, continuous, injective map $\theta \colon (-1,1)^2 \to X$ such that $\theta (0,0) = a$. Then $\im {\theta}$ is not open in $X$.  Now $X^o \cong (-1,1)^2$ via a definable homeomorphism $\psi$, so $\psi \circ \theta \colon (-1,1)^2 \to R^2$ yields a definable, injective, continuous map which does not map $(-1,1)^2$ onto an open subset of $R^2$, a contradiction with Theorem \ref{invarofdom}.

\smallskip\noindent
We assume from now on that
$G^{-1}(t)\cong I$ for all $t\in [\delta ,q]$.  The case when $G^{-1}(t)\cong S^1$ is similar, but simpler, and left to the reader.

\smallskip\noindent
{\em Construction of $\epsilon_0$ good for $G$:\/}
Let $\{\phi_t :\; t\in [\delta ,q]\}$ be a definable family of
homeomorphisms $\phi_t \colon [0,1] \to  G^{-1}(t)$.  For $t\in (\delta ,q)$ we set
$$|\nabla G (\phi_{t}(0))| :=
\lim_{s \to 0} |\nabla G (\phi_{t} (s))|, \mbox{ and } |\nabla G
(\phi_{t}(1))| := \lim_{s \to 1} |\nabla G (\phi_{t} (s))|,$$
where
$$\nabla G (\phi_t (s) ) =
(\frac{\partial G}{\partial x_1 } (\phi_{t}(s)), \frac{\partial
G}{\partial x_2 } (\phi_t (s))),$$ for $s\in (0,1)$. The above limits exist in $R \cup \{ - \infty , \infty  \}$ by o-minimality.
Now set
\begin{eqnarray*}
x_t &:=& \min \{ |\nabla G(x)|: \; x\in G^{-1}(t)  \},\\
H_t &:=& \{ x\in G^{-1}(t) : \; |\nabla G(x)| = x_t \},
\end{eqnarray*}
where $t\in (\delta ,q)$,
and let $H := \bigcup_{t\in (\delta ,q)} H_t$.

By
definable choice, and after replacing $\delta$ by a bigger $\delta \in \ma^{>0}$, and replacing $q$ by a possibly smaller $q>\ma$, we can find a definable
map $\gamma \colon [\delta ,q] \to H$ with $G(\gamma (t)) =t$ for every $t\in [\delta ,q]$.
Next, find $\epsilon' \in \ma^{>0}$
good for $G|_{\im{\gamma}}$, and $\epsilon'' \in \ma^{>0}$ good for
$G|_{\text{bd} X}$, and set
$$\epsilon_0 := \max \{\epsilon' ,\epsilon'' \}.$$
The rest of the proof consists of showing that $\epsilon_0$ is indeed good for $G$.

First, assume towards a contradiction that $\epsilon \in
\ma^{> \epsilon_0 }$ and $a\in G^{-1} (\epsilon )$ are such that $d(a,
G^{-1}(\epsilon_0 )) > \ma$. By continuity of $G$, we may as well assume
that $a \not\in \text{bd} X$.
We now proceed to define what we called in the remark just before Lemma \ref{main lemma} ``a long, steep path down the hill".

\smallskip\noindent {\em Construction of $\alpha$:\/}  
It is clear that $\nabla G$ is defined and
continuous on $X^o= (G^{-1}([\delta ,q]))^o$.
Let $i$ be such that $-\frac{\nabla G(x)}{|\nabla G(x)|} \in S_i$ on $X^o$, and
set $w := -\frac{\nabla G (a )}{|\nabla G(a
)|}$.
Since $a \not\in \text{bd}X$, for all sufficiently small $t>0$, $a + tw \in (G^{-1}([\epsilon_0 , \epsilon ]))^o$.
Let $t_0$ be the first $t>0$ such that $a + t_0 w \in \text{bd} G^{-1}([\epsilon_0 , \epsilon ])$.  We
set $b := a + t_0 w$, and define $\alpha \colon [0,1] \to X$ by $$\alpha (t) = (1-t)a + tb .$$
From now on let $\alpha$ be as in ``Construction of $\alpha$".

\smallskip\noindent
{\bf Claim 4. \/}
$G(\alpha )$ is strictly decreasing.

Assume towards a contradiction that we
can find $t_1 \in (0 , 1)$ so that $(G\circ \alpha )'(t_1
) \geq 0$.  By the definition of $\alpha$,
 $G( \alpha (0)) > G( \alpha (t))$ for all
sufficiently small $t>0$.  Then, by the Mean Value Theorem, $(G \circ \alpha )'(t_2 ) =0$ for some $t_2
\in (0, t_1 ]$. But
$$(G\circ \alpha )'(t_2 ) = \nabla G(\alpha (t_2 )) \cdot
D \alpha (t_2 ),$$
hence $\nabla G (\alpha (t_2 ))$ and
$D \alpha (t_2 )$ are orthogonal, a contradiction with the
definition of $\alpha$.


\smallskip\noindent
{\bf Claim 5.\/}
$\text{d}(a , b) > \ma$.  

To see this, first observe that either $b\in G^{-1}(\epsilon_0 )$, or $b\in \text{bd} X$.  If $b\in G^{-1}(\epsilon_0 )$, then the claim is just the assumption that $\text{d}(a, G^{-1}(\epsilon_0 ))>\ma$.  If $b\in \text{bd} X$, then use $\text{d}(a, b ) + \text{d}(b , G^{-1} (\epsilon_0 )) > \ma$ and $\text{d}(b , G^{-1} (\epsilon_0 )) \in \ma^{>0}$ (since $\epsilon_0$ is good for $G|_{\text{bd} X}$).

\smallskip\noindent {\em Construction of $\beta$:\/}
Note that $|\nabla G(x)|>0$ for all $x\in \im{\alpha|_{[0,1)}}$, since $\im{\alpha|_{[0,1)}} \subseteq X^o$.  If $|\nabla G(b)|=0$, then replace $b$ by an element in $\im \alpha$ which is infinitely close to the original $b$.  Now define $$m_0 := \min\{   |\nabla G(x)|:\; x \in \im{\alpha} \}>0.$$  For $t\in [\epsilon_0 , \epsilon ]$ let
$$
K_t := \{ x\in G^{-1}(t) \cap X^o :\; |\nabla G(x)| - x_t <m_0 \mbox{ and }\text{d}(x, \gamma (t)) <\epsilon_0  \},$$ note that every $K_t$ is nonempty, and put $K := \bigcup_{t\in [\epsilon_0 , \epsilon ]} K_{t}$.
By curve selection, we can find a definable map $\beta \colon [\epsilon_0 , \epsilon ] \to K$ with $G(\beta (t))=t$.

Let $\epsilon_1 , \epsilon_2 \in [\epsilon_0 ,\epsilon ]$ be such that $$G(b)\leq \epsilon_1 < \epsilon_2 \leq \epsilon ,$$ $\text{d}(c,d) > \ma$, for $c, d \in \im{\alpha}$ with $G(c)=\epsilon_1$, $G(d)=\epsilon_2$, and $\beta |_{(\epsilon_1 , \epsilon_2 )}$ is $C^1$.
By Lemma 10.8, p.204 in \cite{hlimit}\footnote{Lemma 10.8 in \cite{hlimit} is an adaptation of Kurdyka and Raby's subanalytic Proposition 1.4 in \cite{kurdyka}.
In \cite{hlimit} it is stated for an elementary extension of an o-minimal expansion of the real field, but the proof goes through for any
o-minimal field.}
we may further assume that $s$ is an orthogonal linear transformation of $R^2$ such that $s(\im{\beta |_{(\epsilon_1 ,\epsilon_2 )}})=\Gamma \phi$, where $\phi \colon p^{2}_{1}s(\im{\beta |_{(\epsilon_1 ,\epsilon_2 )}}) \to R$ is $C^1$, and $|\frac{\partial \phi }{\partial x_1}|<1$ on $p^{2}_{1}s(\Gamma \phi )$.

Let $x\in (c,d)$ and $t=G(x)$.  Then, by the definition of $\alpha$,
\begin{equation}\label{eq}
\frac{|\nabla G(x)|}{\sqrt{2}}\leq |\nabla G(x) \cdot w|\leq |\nabla G(x)|.
\end{equation}
By the definition of $\beta$,
$$|\nabla G(\beta (t))|-m_0 \leq x_t \leq |\nabla G(x)|,$$ so
$$|\nabla G (\beta (t))| \leq |\nabla G(x)| + m_0 , $$ and thus
$$|\nabla G(\beta (t))| \leq 2|\nabla G(x)|, $$ and, using inequality (\ref{eq}), we obtain
\begin{equation}\label{ineq}
|\nabla G(\beta (t)) \cdot \tau | \leq 2 \sqrt{2} |\nabla G(x) \cdot w|,
\end{equation}
 where $\tau $ is the unit tangent vector to $\im{\beta }$ at $\beta (t)$.

\smallskip\noindent
We now use $\alpha$ and $\beta$ to define two one-variable functions, $h_{\alpha}$ and $h_{\beta}$ respectively, which behave in a non-permissible way.

Let $h_{\beta }\colon p^{2}_{1} \Gamma \phi \to (\epsilon_{1}, \epsilon_{2})$ be the function $$h_{\beta }(x) = G (s^{-1} (x, \phi (x))).$$
By the definition of $\beta$, $\epsilon_0$ is good for $G|_{\im{\beta }}$, hence any two points in $\im{\beta }$ are infinitely close to each other.  It follows that $|p^{2}_{1} \Gamma \phi | \in \ma^{>0}$. 
For $x \in p^{2}_{1} \Gamma \phi$ we have
$$|h_{\beta }'(x)|= |\nabla G(s^{-1}(x , \phi (x))) \cdot Ds^{-1} (x , \phi (x )) \cdot D(id_R , \phi)(x)   |.$$  Since $Ds^{-1} (x,\phi (x)) \cdot D(id_R , \phi ) (x)$ is just $c \tau$, where $c\in [1,\sqrt{2})$, and $\tau$ is the unit tangent vector to $\im{\beta }$ at $s^{-1}(x, \phi (x))$, it follows that for every $x \in p^{2}_{1} \Gamma \phi$ there is $c\in [1,\sqrt{2})$ such that
\begin{equation}\label{eq1}
|h_{\beta }'(x)|=c |\nabla G (s^{-1} (x,\phi (x))) \cdot \tau |.
 \end{equation}

Next, note that either the projection $p_1 \colon R^2 \to R \colon (x_1 , x_2 ) \mapsto x_1$ or the projection $p_2 \colon R^2 \to R \colon (x_1 , x_2 ) \mapsto x_2 $ maps
$(c,d)$ onto an open interval of finite, noninfinitesimal length.  Say $|p_1 (c,d)| > \ma$ (the other case is similar), and assume for simplicity $c_1 < d_1$. Define $h_{\alpha}\colon p_1 (c, d)\to (\epsilon_1 , \epsilon_2 )$ by $$h_{\alpha }(x) = G((p_1 |_{(c,d)})^{-1}(x)).$$  Then for all $x\in p_1 (c,d)$,  $$|h_{\alpha }'(x)| = |\nabla G ((p_1 |_{(c,d)})^{-1} (x))\cdot \begin{pmatrix} 1\\ \frac{d_2 - c_2 }{d_1 - c_1 }  \end{pmatrix} |,$$
hence
\begin{equation}\label{eq2}
|h_{\alpha }'(x)| = \lambda |\nabla G ((p_1 |_{(c,d)})^{-1} (x))\cdot w |,
 \end{equation}
 for some $\lambda \in V^{>\ma}$.

Let $x \in (c,d)$ and $y \in \im{\beta }$ be such that $G(x)=G(y)$.  Then inequality (\ref{ineq}) and equations (\ref{eq1}) and (\ref{eq2}) yield that for some $\lambda \in V^{>0}$, $|h_{\beta }'(y)| \leq \lambda |h_{\alpha}'(x)|$.

Then, by Claim 4, the inverses of $h_{\alpha }$ and $h_{\beta }$
 yield two definable, continuous, strictly monotone functions $(\epsilon_{1},\epsilon_2 ) \to R$, with the image of $h_{\alpha}^{-1}$ an interval of length $>\ma$, the image of $h_{\beta}^{-1}$ an interval of infinitesimal length, and such that for every $t\in (\epsilon_{1}, \epsilon_2 )$ there is $\lambda \in V^{>0}$ with
$\lambda |(h_{\beta}^{-1}(t))'| \geq | (h_{\alpha}^{-1}(t))'|$, a contradiction.

We have shown $\st G^{-1}(\epsilon )\subseteq \st G^{-1} (\epsilon_0 )$.  It is left to see that $\st G^{-1}(\epsilon_0 ) \subseteq \st G^{-1}(\epsilon )$.  Assume towards a contradiction that we can find $x \in G^{-1}(\epsilon_0 )$ such that $\text{d}(x, G^{-1}(\epsilon )) >\ma$.  Then we construct a definable map $[0,1] \to X$ whose image at $0$ is $x$, the image at $1$ lies in $\text{bd} G^{-1}([\epsilon_0 , \epsilon ])$, and the restriction of $G$ to the image of this curve is strictly increasing, just as we constructed $\alpha$, except that we set $w:=\frac{\nabla G(x)}{|\nabla G(x)|}$, and we obtain a contradiction similarly as in the proof of the other inclusion.

\end{proof}
\end{section}

\begin{section}{$\Sigma (1) \Rightarrow \Sigma$}
\begin{definition}
A {\em small path\/} is a definable, continuous map $\gamma \colon [0,r ] \to X$, with definable $X\subseteq R^n$ and $r \in R^{>0}$, such that $\st \gamma (0) = \st \gamma (t)$ for all $t\in [0,r ]$.
\end{definition}

\begin{lemma}\label{red}
Let $C=(f,g) \subseteq V^2$ be a cell of dimension two, and let $a \in \partial C$.  Then there is $a' \in C$ and a small path $\gamma \colon [0,r] \to \closure{C}$ such that $\gamma (0) =a$, $\gamma (r) = a'$, and $\gamma (0,1] \subseteq C$.
\end{lemma}
\begin{proof}
By curve selection, we can find a definable and continuous map $\gamma \colon [0,r] \to X$, for some $r >0$, such that $\gamma (0) =a$, and $\text{d}(a, \gamma (t)) = t$ for every $t\in [0,r]$.  We may assume that $r \in \ma^{>0}$, and we set $a' = \gamma (r )$.  Then $a'$ and $\gamma $ have the required properties.
\end{proof}

\begin{lemma}\label{ex of small paths}
Let $C=(f,g) \subseteq V^2$ be a cell of dimension two, and assume that
$f$, $g$ are $C^1$ and
$f'$, $g'$ have constant sign on $p^{2}_{1}C$.
\begin{trivlist}
\item[a)] If
$a,b \in C$ are such that $\st a =\st b$, then there is a small path $\gamma \colon [0,1] \to C$ with
$\gamma (0)=a$ and $\gamma (1) = b$.
\item[b)] If $a,b \in \closure{C}$ are such that $\st a = \st b$, then there is a small path $\gamma \colon [0,1] \to \closure{C}$ with $\gamma (0)=a$, $\gamma (1)=b$, and such that $\gamma (0,1) \subseteq C$.
\end{trivlist}
\end{lemma}

\begin{proof}
First, let $a,b \in C$ with $\st a = \st b$.
If $a_1 = b_1$, then we define $\gamma (t)=(1-t)a+tb$, for $t\in [0,1]$.
So assume that $a_1 \not= b_1$, say $a_1 < b_1$ and define
$\gamma \colon [a_1 , b_1 ] \to I^2$ by $\gamma (t) = (t, \frac{b_2 - a_2 }{b_1 - a_1 }(t-a_1 ) +a_2 )$. Let
$$a_1 < t_1 < \dots < t_k < b_1$$ be such that
$$
\gamma
(t_{1} , t_{2}),  \gamma (t_3 , t_4 ), \dots , \gamma (t_{k-1}, t_k )$$ have empty intersection with $C$, and
$$\gamma (a_1, t_{1}), \gamma (t_2 , t_3 ) , \dots , \gamma (t_k , b_1 )$$ are subsets of $C$. Note that then $k$ is even, and let $i \in \{1, \dots ,\frac{k}{2} \}$.  Then either
 $\gamma (t_{2i-1} ),\gamma (t_{2i}) \in \closure{\Gamma f}$, or
 $\gamma (t_{2i-1}) , \gamma (t_{2i}) \in \closure{\Gamma g}$.  Assume $\gamma (t_{2i-1} ),\gamma (t_{2i}) \in \closure{\Gamma f}$ (the other case is similar).  Note that $\st a_2 = \st f(t_{2i-1 })=\st f(t)$ for every $t \in [t_{2i-1}, t_{2i}]$, since $f'$ has constant sign on $p^{2}_{1}C$.
 Find $\delta \in \ma^{>0}$ such that 
$$\delta < \min \{ t_{2i-1} - t_{2i-2}, t_{2i+1} - t_{2i} \},$$ and
$$\gamma_2 (t_{2i-1} - \delta )-f(t_{2i-1}-\delta ) \in \ma^{>0}, \;
\gamma_2 (t_{2i} + \delta ) - f(t_{2i} + \delta )\in \ma^{>0}.$$
Then $g(x) - f(x) > 0 $ on $[t_{2i-1} - \delta , t_{2i } + \delta ]$.  Take $\epsilon \in \ma^{>0}$ such that $$\epsilon < \min \{ g(x)-f(x):\;x\in  [t_{2i-1} - \delta , t_{2i} + \delta ]  \}.$$
Define $\alpha_i \colon [t_{2i - 1} - \delta , t_{2i} + \delta ] \to C$ by
\begin{eqnarray*}
\alpha_{i} (t) &=& (t , f (t) + \epsilon ) \mbox{ for } t\in (t_{2i-1}-\frac{\delta }{2}, t_{2i}+ \frac{\delta}{2}),\\
\alpha_{i} (t) &=& (t, (1-\frac{2}{\delta }(t-t_{2i-1} + \delta )) \gamma_2 (t_{2i -1} -\delta ) + \frac{2}{\delta }(t-t_{2i-1} + \delta ) (f(t_{2i-1} - \frac{\delta}{2} )+\epsilon ))\\ &&\mbox{when } t\in [t_{2i-1}-\delta , t_{2i-1} - \frac{\delta }{2}],\\
\alpha_{i} (t) &=& (t,(1-\frac{2}{\delta }(t-t_{2i} - \frac{\delta}{2} )) (f(t_{2i} + \frac{\delta}{2} )+\epsilon ) + \frac{2}{\delta }(t-t_{2i} - \frac{\delta}{2} )\gamma_2 (t_{2i} +\delta ))\\ &&\mbox{whenever } t\in [t_{2i}+ \frac{\delta }{2} , t_{2i} + \delta ].\\
\end{eqnarray*}
(So $\alpha_i |_{[t_{2i-1}-\delta , t_{2i-1} - \frac{\delta }{2}]}$ is a parametrization of the line segment $$[\gamma_2 (t_{2i -1} -\delta ), f(t_{2i-1} - \frac{\delta}{2} )+\epsilon],$$ and $\alpha_i |_{[t_{2i}+ \frac{\delta }{2} , t_{2i} + \delta ]}$ is a parametrization of $[f(t_{2i} + \frac{\delta}{2})+\epsilon , \gamma_2 (t_{2i} +\delta )]$.)
For every $i \in \{ 1, \dots , \frac{k}{2}  \}$ replace $\gamma |_{[t_{2i} - \delta , t_{2i} + \delta ]}$ by $\alpha_i$.  The resulting map is, after a linear change of variables, as required.

To prove the second part of the lemma, use Lemma \ref{red} to find $a', b' \in C$ and small paths $\gamma_1 \colon [0,r] \to \closure{C}$, $\gamma_2 \colon [0,r] \to \closure{C}$ such that $\gamma_1 (0)=a$, $\gamma_1 (r) = a'$, $\gamma_2 (0)=b$, $\gamma_2 (r)=b'$, and $\gamma_1 (0,r] \subseteq C$, $\gamma_2 (0,r] \subseteq C$.  Then use the first part of the lemma to find a small path $\gamma_3 \colon [0,1] \to C$ with $\gamma_3 (0)=a'$, $\gamma_3 (1)=b'$, and $\im{\gamma}\subseteq C$. After a linear changes of variables, $\gamma_1 \colon [0,\frac{1}{3}]\to \closure{C}$, $\gamma_3 \colon [\frac{1}{3}, \frac{2}{3}]\to C$, and $\gamma_2 \colon [\frac{2}{3},1]\to \closure{C}$.  Then define $\gamma \colon [0,1]\to \closure{C}$ by $\gamma (t)=\gamma_1 (t)$ if $t\in [0,\frac{1}{3}]$, $\gamma (t) = \gamma_3 (t)$ if $t\in (\frac{1}{3},\frac{2}{3})$, and $\gamma (t)=\gamma_2 (t)$ if $t\in [\frac{2}{3},1]$.  Clearly, $\gamma$ satisfies the requirements.
\end{proof}

\begin{lemma}\label{mean value for small paths}
Let $(f,g) \subseteq V^2$ be a cell of dimension two such that $f$,
$g$ are $C^1$ and $f'$, $g'$ have constant sign on $p^{2}_{1}(f,g)$.
Let $h\colon (f,g) \to R$ be definable and $C^1$ and  such that
$|\frac{\partial h}{\partial x_1
}|, |\frac{\partial h}{\partial x_2 }| < 1$. Then $\st h(a) = \st h(b)$ whenever $a,b \in (f,g)$ are such that $\st a=\st
b$.  If moreover $h$ extends continuously to $\cl h \colon \cl (f,g) \to R$, then $\st \cl h (a) = \st \cl h (b)$ whenever $a,b \in \cl (f,g)$ with $\st a = \st b$.
\end{lemma}
\begin{proof}
Let $a,b \in (f,g)$ be such that $\st (a)=\st
(b)$.  Then, by Lemma \ref{ex of small paths}, we can find a small path $\gamma \colon
[0,1] \to (f,g)$ with $\gamma (0)=a$ and $\gamma (1) = b$.

By Lemma 10.8 in \cite{hlimit}, we can partition $\im{\gamma }$ into definable
$X_1 ,\dots,X_{k+1}$, where $\dim{X_{k+1}} = 0$, and we can find
orthogonal linear transformations $s_1 , \dots , s_k$ of $R^2$, such
that for $i=1,\dots ,k$, $s_i X_i $ is the graph of a definable
$C^1$-function $\phi_i \colon p^{2}_{1}s_i X_i \to V$ on an open
$p^{2}_{1} s_i X \subseteq V$ with $|\phi_i '|< 1$ on $p^{2}_{1} s_i
X_i$.
We may assume that each $\Gamma \phi_{i}$ is a cell, and we let $\cl \phi_i$ be the continuous extension of $\phi_i$ to $\closure{p^{2}_{1} s_i X_i} \to R$.

Define $H_i \colon \cl{( p^{2}_{1}s_i
X_i )} \to R$ by $$H_i (x) = h(s_{i}^{-1} (x, \cl \phi_i  (x))).$$ Assume towards a contradiction that $\st
h(a) \not= \st h(b)$. For $i=1,\dots ,k$, let $x_i
, y_i \in R$ be such that $(x_i , y_i ) = p^{2}_{1} s_i X_i$.  Then for some $i$, $\st H_i (x_i ) \not= \st H_i (y_i )$.  However, $|x_i - y_i | \in \ma^{>0}$ (as $s_i$ is an isometry), and
$$|H'_{i}(t)| = |\nabla h
(s_{i}^{-1}(t , \phi_i (t))) \cdot D s_{i}^{-1} (t,\phi_i (t)) \cdot
D (id_R , \phi_i ) (t)| \in V^{>0},$$ for $t\in (x_i , y_i )$, a
contradiction with the mean value theorem.

Assume now that $h$ extends continuously to $\cl h \colon \cl (f,g) \to R$ and let $a,b \in \cl (f,g)$ be such that $\st a = \st b$.  By continuity of $\cl h$
and the proof of Lemma \ref{red} we can find $a',b' \in (f,g)$ infinitely close to $a$ and to $b$ respectively, and small paths $\alpha \colon [0,r] \to \cl (f,g)$, $\beta \colon [0,r] \to \cl (f,g)$, where $r>0$, so that $\alpha (0)=a$, $\alpha (r)=a'$, $\beta (0)=b$, $\beta (r) = b'$, $\alpha (0,r], \beta (0,r] \subseteq (f,g)$, and $$|\cl h(a) - h(\alpha (t))|,\; |\cl h(b)-h(\beta (t))| \in \ma^{>0}$$ for all $t\in (0,r]$.  Now use the first part of the lemma to find a small path $\gamma \colon [0,1] \to (f,g)$ with $\gamma (0)=a'$ and $\gamma (1)=b'$.  It is now obvious how to construct the desired map from $\alpha$, $\beta$, and $\gamma$.
\end{proof}

\begin{lemma}\label{extension lemma}
Let $C=(f,g) \subseteq V^2$ be a cell of dimension two such that $f$, $g$ are $C^1$ and $f'$, $g'$ have constant sign on $p^{2}_{1}C$.
Let further $h\colon C \to R$ be definable and $C^1$, with $|\frac{\partial h}{\partial x_1 }|$, $|\frac{\partial h}{\partial x_2 }| < 1$, and $\Gamma h \subseteq V^3$.  Then $h$ extends continuously to a definable function $\cl h \colon \closure{C} \to R$.
\end{lemma}
\begin{proof}
We first show that $\closure{ \Gamma h}$ is the graph of a function.  By Lemma \ref{mean value for small paths}, $\st \Gamma h$ is the graph of a function, and $\st \Gamma h = \st \closure{\Gamma h}$.  So if there are $x,y \in \closure{ \Gamma h}$ with $p^{3}_{2}x = p^{3}_{2} y$ and $x_3 \not= y_3$, then $|x_3 - y_3 | \in \ma^{>0}$.  Set $\delta := |x_3 - y_3 |$, let  $\ma'$ be the maximal ideal $$\{x \in R:\; |x|<\delta q \mbox{ for all }q\in \mathbb{Q}^{>0}  \},$$ and let $V'$ be the corresponding valuation ring with residue map $\st' \colon V' \to \boldsymbol{k}'$. Then $V\subseteq V'$. Since Lemma \ref{mean value for small paths} holds for any choice of $V$, $\st' \Gamma h = \st' \closure{\Gamma h}$ is again the graph of a function, a contradiction with $\delta \not\in \ma'$.

By Lemma 1.7, p.95 in \cite{book}, $\closure{C} = p^{3}_{2} \closure{ \Gamma h}$.  We let $\cl h$ be the function $\closure{C} \to R$ whose graph is the set $\closure{ \Gamma h}$.  To see that $\cl h$ is continuous, assume towards a contradiction that $a \in \closure{C}$ and $\epsilon >0$ are such that for every $\delta >0$, we can find $x_{\delta} \in \cl C$ with $\text{d}(x_{\delta },a) < \delta $ and $|\cl h(a)- \cl h(x_{\delta })| > \epsilon $.  Then curve selection yields a definable and continuous $\gamma \colon (0,r] \to \closure{C}$, $r>0$, with $\lim_{t\to 0}\gamma (t) = a$, $\text{d}(\gamma (t), a)= t$, and $|\cl h(\gamma (t)) - \cl h (a)|>\epsilon $.  Then, since $\closure{\Gamma h}$ is bounded, $\lim_{t \to 0} \cl h (\gamma (t) )$ exists in $R$, and since $\closure{\Gamma h}$ is closed, $\lim_{t\to 0} \cl h( \gamma (t)) \in \Gamma \cl h$, a contradicition with $\cl h$ being a function.
\end{proof}

\begin{theorem}
$\Sigma (1) \Rightarrow \Sigma$.
\end{theorem}
\begin{proof}
Assume that $(R,V)\models \Sigma (1)$, and let $X\subseteq I^{1+n}$ be definable.  By Lemmas \ref{red to dim2} and \ref{good for dim 1}, towards proving $\Sigma (n)$, we may assume that $X$ is of dimension two, and $p^{n+1}_{1}X = (0,1)$.

By Lemma 10.8 from \cite{hlimit}, we can find
 a partition of $X$ into definable sets $X_1 , \dots ,X_{k+1}$, where
 $\dim X_{k+1} <2$, and orthogonal linear transformations
 $s_1 , \dots ,s_k$ of $R^{n+1}$, such that each $s_i X_i$ is the
 graph
 of a definable $C^1$-map $\phi_i \colon p^{n+1}_{2} s_i X_i \to R^{n-1}$ on an open
 $p^{n+1}_{2} s_i X_i \subseteq R^2$ and
 $|\frac{\partial \phi_{ij} }{\partial x_l}| < 1$ on $p^{n+1}_{2}s_i X_i$ for $l = 1,2$ and $j=1,\dots ,n-1$.

Let $i\in \{1,\dots ,k  \}$, and set, for the sake of simplicity, $X:=X_i$, $s:=s_i$, and $\phi := \phi_i$.
By cell decomposition, we may assume that $s X$ is a cell such that $p^{n+1}_{2} sX = (g,h)$, where $g,h \colon p^{2}_{1} (g,h) \to R$ are definable and $C^1$, and $g'$, $h'$ have constant sign on $p^{2}_{1}(g,h)$.

By Lemma \ref{extension lemma}
each coordinate function $\phi_{j}$ of $\phi$ extends continuously to $$\cl \phi_{j} \colon \closure{p^{n+1}_{2} s X} \to R.$$ We set $$\cl \phi  := (\cl \phi_{1}, \dots ,\cl \phi_{n-1} ),$$ and one easily checks, using $\closure{ \Gamma \phi_j }=\Gamma (\cl \phi_j )$ and the continuity of $\cl \phi_j$ for all $j$, that $$\closure{X} = s^{-1} (\closure{sX} ) = s^{-1}( \Gamma \cl \phi ).$$
Let $F\colon \closure{X} \to [0,1]$ be the coordinate projection $(x_1 , \dots ,x_{n+1}) \mapsto (x_1 )$, and
define $G \colon \closure{ p^{n+1}_{2} s X } \to R^{\geq 0}$ by
$$G (x):=F( s^{-1}(x,
\cl \phi (x) )).$$  Since $G$ is continuous, Lemma \ref{main lemma} yields $\epsilon_0$ good for $G$. We claim that $\epsilon_0$ is also good
for $F$.

To see that $\st F^{-1}(\epsilon )\subseteq \st F^{-1}(\epsilon_0 )$, let $x \in F^{-1}(\epsilon )$, for some $\epsilon \in \ma^{>\epsilon_0 }$.  Our aim is to show that $\text{d}(x, F^{-1}(\epsilon_0 )) \in \ma^{>0}$.
We have $p^{n+1}_{2} s(x) \in G^{-1} (\epsilon )$, hence $\text{d}(p^{n+1}_{2}s (x) , G^{-1}(\epsilon_0 )) = \delta$  for some $\delta \in \ma^{>0}$.  Take $y \in G^{-1} (\epsilon_0 )$ such that $\text{d}(p^{n+1}_{2}s (x),y) = \delta$.
Then, by the assumption on $g'$ and $h'$, and since $|\frac{\partial \phi_j }{\partial x_1 }|,|\frac{\partial \phi_j }{\partial x_2 }|< 1$ for $j=1,\dots ,n-1$, Lemma \ref{mean value for small paths} applied to each coordinate function of $\phi$ yields $\text{d}(s x,(y,\cl \phi (y))) \in \ma^{> 0}$. Now $F(s^{-1}(y, \cl \phi (y)) = \epsilon_0$, and, since $s$ is an isometry, $$\text{d}(x, s^{-1} (y,\cl \phi (y) ) ) \in \ma^{> 0}.$$  Similarly, one shows that $\st F^{-1}(\epsilon_0 )\subseteq \st F^{-1}(\epsilon )$.

By Lemma \ref{good for dim 1}, we can find $\epsilon_{k+1}$ good for $X_{k+1} \to (0,1)\colon x\mapsto x_1$, and by the above, we can find $\epsilon_i$ good for $X_i \to (0,1) \colon x \mapsto x_1$ for each $i=1,\dots ,k$.
Then $\epsilon_0 := \max_{i=1,\dots, k+1}\{ \epsilon_i \}$ is such that $\st X (\epsilon_0 ) = \st X (\epsilon )$  for all $\epsilon \in \ma^{>\epsilon_0 }$.

\smallskip\noindent

\end{proof}
Theorem \ref{maintheorem} from the introduction now follows.

\begin{ack}
An earlier version of the proof of Lemma \ref{main lemma} uses a consequence of Shiota's Hauptvermutung for o-minimal fields (see \cite{shiota}), instead of the Invariance of Domain Theorem.  The present version is shorter, but nevertheless inspired by Shiota's striking result.

 I thank Lou van den Dries for his interest in and comments on this paper.
\end{ack}

\end{section}

\end{document}